\documentclass{birkjour}
 \newtheorem{thm}{Theorem}[section]
 
 \newtheorem{lem}[thm]{Lemma}
 
 \theoremstyle{definition}
 
 \theoremstyle{remark}

 \numberwithin{equation}{section}

 \newcommand{\arcangle}{%
  \mathord{<\mspace{-9mu}\mathrel{)}\mspace{2mu}}%
}

  \usepackage{cite}
  \usepackage{float}

\usepackage{mathrsfs}
\usetikzlibrary{arrows}
\usepackage{pgfplots}
  
\usepackage{hyperref}
\hypersetup{
	colorlinks,
	linkcolor = {blue},
	citecolor = {blue},
	filecolor = {blue},
	urlcolor = {blue} 
} 

\usepackage{subcaption}

\begin{document}

%
%
%
%
%
%
%
%
%

\title[The Seven Circles Theorem for Arbitrary Closed Six-Circle Chains]
 {When Points Become Circles: The Seven Circles Theorem for Arbitrary Closed Six-Circle Chains}

\author[Miłosz Płatek]{Miłosz Płatek}

\address{Independent researcher\\ Krakow, Poland\\Current affiliation: University of California, Berkeley \\Berkeley, CA, USA}

\email{milosz\_platek@berkeley.edu}




\begin{abstract}
    The Seven Circles Theorem states that if six circles form a closed chain and are tangent to a common circle, then the three lines joining opposite points of tangency on the common circle are concurrent. We extend this result to an arbitrary closed six-circle chain. Given such a chain, we consider two circles, each tangent to one of the two alternating triples of circles in the chain. For each circle in the chain, we then construct a circle tangent to it and to both of these circles. We prove that the three lines joining the centers of the newly constructed circles corresponding to opposite members of the chain are concurrent. In the classical configuration, the two circles associated with the alternating triples coincide with the common circle, while the six newly constructed circles degenerate to the six points of tangency, viewed as circles of radius zero. Thus, the classical Seven Circles Theorem is recovered as a degenerate case of our generalization.

\end{abstract}

\maketitle

\section{Introduction}

The Seven Circles Theorem gives a striking concurrence in a simple configuration of tangent circles.

\begin{thm}[Seven Circles Theorem]
Let $\omega_1,\ldots,\omega_6$ form a closed chain, with each circle tangent to its two neighbors, and suppose that all six circles are tangent to a common circle $\Omega$ at points $S_1,\ldots,S_6$, respectively. Then the lines $S_1S_4$, $S_2S_5$, and $S_3S_6$ are concurrent; see Figure~\ref{basic}.
\end{thm}

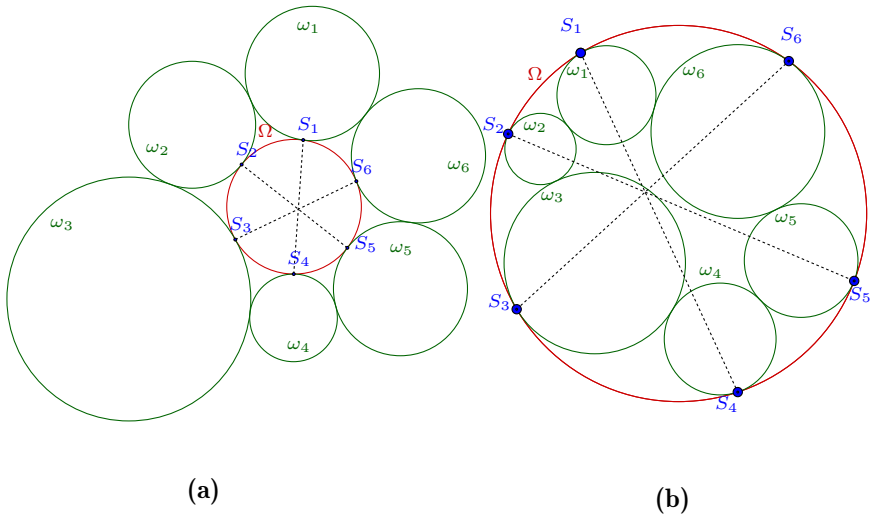
\begin{figure}[ht]
    \centering

    \begin{minipage}[t]{0.47\textwidth}
        \vspace{0pt}
        \centering
        \pgfplotsset{compat=1.15}
\definecolor{qqwuqq}{rgb}{0,0.39215686274509803,0}
\definecolor{ccqqqq}{rgb}{0.8,0,0}
\definecolor{qqqqff}{rgb}{0,0,1}
\begin{tikzpicture}[scale=0.215][line cap=round,line join=round,>=triangle 45,x=1cm,y=1cm]
\clip(-19.135779007623274,-12.759592993186665) rectangle (11.77783493422913,13.639515435599671);
\draw [line width=0.4pt,color=ccqqqq] (-0.78,0.91) circle (4.14004830889689cm);
\draw [line width=0.4pt,color=qqwuqq] (-0.8130227655199813,-5.925712462636118) circle (2.6957439184079544cm);
\draw [line width=0.4pt,color=qqwuqq] (5.765324852054217,-4.132764998517105) circle (4.122563641884037cm);
\draw [line width=0.4pt,color=qqwuqq] (6.871167184926429,4.043120306481405) circle (4.127768945166069cm);
\draw [line width=0.4pt,color=qqwuqq] (0.344358482040811,9.10542364719416) circle (4.132143102883082cm);
\draw [line width=0.4pt,color=qqwuqq] (-7.0473880274231195,5.952439104590497) circle (3.903977101915168cm);
\draw [line width=0.4pt,color=qqwuqq] (-10.94249369507422,-4.76564972799046) circle (7.499937856784813cm);
\draw [line width=0.4pt,dash pattern=on 1pt off 1pt] (-0.8,-3.23)-- (-0.21728352011558283,5.011627745574493);
\draw [line width=0.4pt,dash pattern=on 1pt off 1pt] (-4.394541653022943,-1.1086848784696857)-- (3.051265349929664,2.478886808673053);
\draw [line width=0.4pt,dash pattern=on 1pt off 1pt] (2.4995877679293326,-1.616718043717639)-- (-4.005659775919862,3.5052107833493245);
\begin{scriptsize}
\draw [fill=qqqqff] (-0.8,-3.23) circle (2.5pt);
\draw[color=qqqqff] (-0.4758228869501947,-2.2902397319334513) node {$S_4$};
\draw[color=qqwuqq] (-0.4758228869501947,-7.6902397319334513) node {$\omega_4$};
\draw [fill=qqqqff] (2.4995877679293326,-1.616718043717639) circle (2.5pt);
\draw[color=qqqqff] (3.6348144892982545,-1.3994139797101704) node {$S_5$};
\draw[color=qqwuqq] (5.8348144892982545,-1.6994139797101704) node {$\omega_5$};
\draw [fill=qqqqff] (3.051265349929664,2.478886808673053) circle (2.5pt);
\draw[color=qqqqff] (3.3937532671064345,3.3851386273496047) node {$S_6$};
\draw[color=qqwuqq] (9.3937532671064345,3.3851386273496047) node {$\omega_6$};
\draw [fill=qqqqff] (-0.21728352011558283,5.011627745574493) circle (2.5pt);
\draw[color=qqqqff] (0.12611118145861427,5.9218607727867285) node {$S_1$};
\draw[color=qqwuqq] (0.12611118145861427,11.9218607727867285) node {$\omega_1$};
\draw [fill=qqqqff] (-4.005659775919862,3.5052107833493245) circle (2.5pt);
\draw[color=qqqqff] (-3.6574743913967565,4.417025601764705) node {$S_2$};
\draw[color=qqwuqq] (-9.1574743913967565,4.417025601764705) node {$\omega_2$};
\draw [fill=qqqqff] (-4.394541653022943,-1.1086848784696857) circle (2.5pt);
\draw[color=qqqqff] (-4.0444320068024195,-0.18347049250261982) node {$S_3$};
\draw[color=qqwuqq] (-15.0444320068024195,-0.18347049250261982) node {$\omega_3$};
\draw[color=ccqqqq] (-2.5444320068024195,5.58347049250261982) node {$\Omega$};
\end{scriptsize}
\end{tikzpicture}

        \smallskip
        \textbf{(a)}
    \end{minipage}
    \hfill
    \begin{minipage}[t]{0.47\textwidth}
        \vspace{0pt}
        \centering
        \pgfplotsset{compat=1.15}
\definecolor{wewdxt}{rgb}{0.43137254901960786,0.42745098039215684,0.45098039215686275}
\definecolor{qqqqff}{rgb}{0,0,1}
\definecolor{qqwuqq}{rgb}{0,0.39215686274509803,0}
\definecolor{ccqqqq}{rgb}{0.8,0,0}
\begin{tikzpicture}[scale=0.7][line cap=round,line join=round,>=triangle 45,x=1cm,y=1cm]
\clip(-2.586841066045264,-4.748365035483542) rectangle (5.359783611741766,3.5287951574973557);
\draw [line width=0.4pt,color=ccqqqq] (1.4574282727260166,-0.5091548261422036) circle (3.553477170322049cm);
\draw [line width=0.4pt,color=ccqqqq] (1.4574282727260166,-0.5091548261422036) circle (3.5534771703220493cm);
\draw [line width=0.4pt,color=qqwuqq] (0.09444736030275758,1.7265250022188288) circle (0.9350842925112046cm);
\draw [line width=0.4pt,color=qqwuqq] (3.7726639060404343,-1.397693672656407) circle (1.0735946490050405cm);
\draw [line width=0.4pt,color=qqwuqq] (-0.12659836136365832,-1.4453490487768734) circle (1.713477170322049cm);
\draw [line width=0.4pt,color=qqwuqq] (-1.1519827582655342,0.7099698446723841) circle (0.673322505785931cm);
\draw [line width=0.4pt,color=qqwuqq] (2.577854005114257,1.0402606130274055) circle (1.6413992365194365cm);
\draw [line width=0.4pt,color=qqwuqq] (2.2428756801250356,-2.879616439170117) circle (1.056275558791911cm);
\draw [line width=0.4pt,dash pattern=on 1pt off 1pt] (-0.39230237576631954,2.5249342509990917)-- (2.575107124843546,-3.8822830453964365);
\draw [line width=0.4pt,dash pattern=on 1pt off 1pt] (-1.6017035106542763,-2.3171683028376915)-- (3.5396693732051228,2.370336795563883);
\draw [line width=0.4pt,dash pattern=on 1pt off 1pt] (-1.7620108222603745,0.9949767883361316)-- (4.774979364835675,-1.7823613102766558);
\begin{scriptsize}
\draw[color=qqwuqq] (-0.43836021728481933,2.1584500007559257) node {$\omega_1$};
\draw[color=qqwuqq] (3.5005213387759957,-0.6235572521261727) node {$\omega_5$};
\draw[color=qqwuqq] (-0.9203911769426114,-0.19661554500070225) node {$\omega_3$};
\draw [fill=qqqqff] (-1.6017035106542763,-2.3171683028376915) circle (2.5pt);
\draw[color=qqqqff] (-1.925770035657435,-2.1798286361641788) node {$S_3$};
\draw [fill=qqqqff] (-0.39230237576631954,2.5249342509990917) circle (2.5pt);
\draw[color=qqqqff] (-0.5623110354825372,3.039878041273026) node {$S_1$};
\draw [fill=qqqqff] (4.774979364835675,-1.7823613102766558) circle (2.5pt);
\draw[color=qqqqff] (4.891524965217054,-2.0421055048333816) node {$S_5$};
\draw[color=qqwuqq] (-1.250926692136526,1.1255265157749486) node {$\omega_2$};
\draw[color=ccqqqq] (-1.250926692136526,2.1255265157749486) node {$\Omega$};
\draw [fill=qqqqff] (-1.6017035106542763,-2.3171683028376915) circle (0.5pt);
\draw[color=qqwuqq] (1.758625187011748,2.1584500007559257) node {$\omega_6$};
\draw[color=qqwuqq] (2.0682007729356995,-1.6223514138447952) node {$\omega_4$};
\draw [fill=qqqqff] (-0.39230237576631954,2.5249342509990917) circle (0.5pt);
\draw [fill=qqqqff] (4.774979364835675,-1.7823613102766558) circle (0.5pt);
\draw [fill=qqqqff] (-1.7620108222603745,0.9949767883361316) circle (2.5pt);
\draw[color=qqqqff] (-2.0338501771175093,1.1602587581779922) node {$S_2$};
\draw [fill=qqqqff] (-1.7620108222603745,0.9949767883361316) circle (0.5pt);
\draw [fill=qqqqff] (2.575107124843546,-3.8822830453964365) circle (2.5pt);
\draw[color=qqqqff] (2.3711916618634543,-4.107952474795336) node {$S_4$};
\draw [fill=qqqqff] (3.5396693732051228,2.370336795563883) circle (2.5pt);
\draw[color=qqqqff] (3.5969275307075543,2.86083797054299) node {$S_6$};
\draw [fill=qqqqff] (3.5396693732051228,2.370336795563883) circle (0.5pt);
\draw [fill=qqqqff] (2.575107124843546,-3.8822830453964365) circle (0.5pt);
\draw [fill=qqqqff] (-0.39230237576631954,2.5249342509990917) circle (2.5pt);
\end{scriptsize}
\end{tikzpicture}

        \smallskip
        \textbf{(b)}
    \end{minipage}

    \caption{The Seven Circles Theorem with the six circles tangent to the common circle (a) externally and (b) internally.}
    \label{basic}
\end{figure}

At first sight, removing the common circle seems to remove the theorem along with it. But does a related concurrence persist for an arbitrary closed chain of six circles?

We show that it does. The key is to reinterpret the classical configuration. We replace the role of the common circle with two circles that coincide in the classical case, and regard the six points of tangency as circles of radius zero.

The Seven Circles Theorem first appeared in the 1974 book by Evelyn, Money-Coutts, and Tyrrell~\cite{evelyn}. Subsequent proofs have revealed different structures underlying the concurrence. Cundy and Rabinowitz gave elementary proofs~\cite{cundy,rabin}, Brown related the theorem to Brianchon's theorem~\cite{brown}, and Drach and Schwartz gave a hyperbolic interpretation~\cite{hyperbolic}. These approaches retain the common circle, whereas we begin with an arbitrary closed six-circle chain.

\section{From a Six-Circle Chain to the General Configuration}

Let $\omega_1,\ldots,\omega_6$ form an arbitrary closed chain, indexed in cyclic order, as shown in Figure~\ref{any_chain}.

\begin{figure}[ht]
    \centering
    \pgfplotsset{compat=1.15}
\definecolor{wewdxt}{rgb}{0.43137254901960786,0.42745098039215684,0.45098039215686275}
\definecolor{qqwuqq}{rgb}{0,0.39215686274509803,0}
\begin{tikzpicture}[scale=0.3][line cap=round,line join=round,>=triangle 45,x=1cm,y=1cm]
\clip(-10.32808933268825,-8.263519374882472) rectangle (7.825739420056494,10.395320719445909);
\draw [line width=0.4pt,color=qqwuqq] (0.9237346814822205,3.8254590283722223) circle (1.2874247901002658cm);
\draw [line width=0.4pt,color=qqwuqq] (-0.9587485739145309,1.0553473646319262) circle (1.1703378296922424cm);
\draw [line width=0.4pt,color=qqwuqq] (2.842100279247295,0.28475678550565386) circle (1.2888467480222292cm);
\draw [line width=0.4pt,color=qqwuqq] (-0.10665297235761506,-3.7422306366426854) circle (3.7023227979858664cm);
\draw [line width=0.4pt,color=qqwuqq] (4.779697020900073,3.626144718424509) circle (2.5736853814901486cm);
\draw [line width=0.4pt,color=qqwuqq] (-4.743874904406102,5.406073047853568) circle (4.5964633033874955cm);
\begin{scriptsize}
\draw[color=qqwuqq] (1.0479556251078705,4.56111074693155) node {$\omega_1$};
\draw[color=qqwuqq] (-0.9986692855236743,1.7968381403643825) node {$\omega_3$};
\draw[color=qqwuqq] (2.828785092800254,1.05261090013476) node {$\omega_5$};
\draw[color=qqwuqq] (-3.0718737404491354,-0.43584358032448395) node {$\omega_4$};
\draw[color=qqwuqq] (3.8653873202629847,5.491394797218579) node {$\omega_6$};
\draw[color=qqwuqq] (-5.384294094019842,9.478326441305839) node {$\omega_2$};
\end{scriptsize}
\end{tikzpicture}
    \caption{An arbitrary closed six-circle chain.}
    \label{any_chain}
\end{figure}
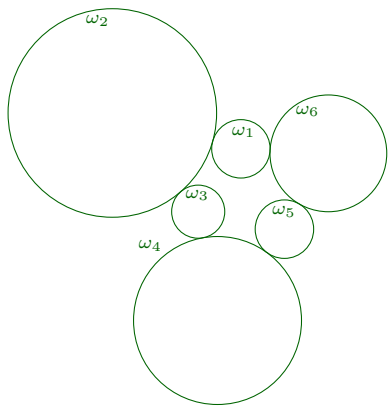

The cyclic order divides the chain naturally into the two alternating triples
\[
\omega_1,\omega_3,\omega_5
\qquad\text{and}\qquad
\omega_2,\omega_4,\omega_6.
\]

Let $\Omega_1$ and $\Omega_2$ be circles tangent to the members of the first and second triples, respectively; see Figure~\ref{opisane}.

The orientation conventions and the precise choices of tangent circles will be specified in the following section.

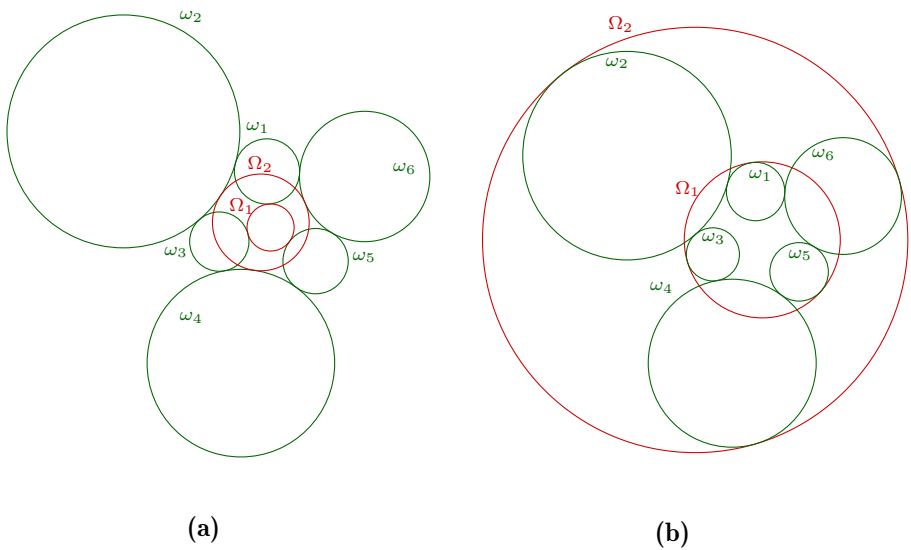
\begin{figure}[ht]
    \centering

    \begin{minipage}[t]{0.47\textwidth}
        \vspace{0pt}
        \centering
        \pgfplotsset{compat=1.15}
\definecolor{ccqqqq}{rgb}{0.8,0,0}
\definecolor{wewdxt}{rgb}{0.43137254901960786,0.42745098039215684,0.45098039215686275}
\definecolor{qqwuqq}{rgb}{0,0.39215686274509803,0}
\begin{tikzpicture}[scale=0.335][line cap=round,line join=round,>=triangle 45,x=1cm,y=1cm]
\clip(-9.78769741060756,-7.844030074907152) rectangle (7.67510624384702,10.298004336495383);
\draw [line width=0.4pt,color=qqwuqq] (0.9237346814822205,3.8254590283722223) circle (1.2874247901002658cm);
\draw [line width=0.4pt,color=qqwuqq] (-0.9587485739145309,1.0553473646319262) circle (1.1703378296922424cm);
\draw [line width=0.4pt,color=qqwuqq] (2.842100279247295,0.28475678550565386) circle (1.2888467480222292cm);
\draw [line width=0.4pt,color=qqwuqq] (-0.10665297235761506,-3.7422306366426854) circle (3.7023227979858664cm);
\draw [line width=0.4pt,color=qqwuqq] (4.779697020900073,3.626144718424509) circle (2.5736853814901486cm);
\draw [line width=0.4pt,color=qqwuqq] (-4.743874904406102,5.406073047853568) circle (4.5964633033874955cm);
\draw [line width=0.4pt,color=ccqqqq] (1.0655646966686945,1.6131537181838893) circle (0.9294221912440292cm);
\draw [line width=0.4pt,color=ccqqqq] (0.6797626122206704,1.8137042156415542) circle (1.908992694346952cm);
\begin{scriptsize}
\draw[color=qqwuqq] (0.5646807439511828,5.550076438777027) node {$\omega_1$};
\draw[color=qqwuqq] (-2.0646807439511828,9.950076438777027) node {$\omega_2$};
\draw[color=qqwuqq] (-2.6646807439511828,0.650076438777027) node {$\omega_3$};
\draw[color=qqwuqq] (-2.0646807439511828,-1.950076438777027) node {$\omega_4$};
\draw[color=qqwuqq] (4.7646807439511828,0.450076438777027) node {$\omega_5$};
\draw[color=qqwuqq] (6.3646807439511828,3.950076438777027) node {$\omega_6$};
\draw[color=ccqqqq] (0.6646807439511828,4.1076438777027) node {$\Omega_2$};
\draw[color=ccqqqq] (-0.0646807439511828,2.450076438777027) node {$\Omega_1$};
\end{scriptsize}
\end{tikzpicture}

        \smallskip
        \textbf{(a)}
    \end{minipage}
    \hfill
    \begin{minipage}[t]{0.47\textwidth}
        \vspace{0pt}
        \centering
        \pgfplotsset{compat=1.15}
\definecolor{wewdxt}{rgb}{0.43137254901960786,0.42745098039215684,0.45098039215686275}
\definecolor{qqwuqq}{rgb}{0,0.39215686274509803,0}
\definecolor{ccqqqq}{rgb}{0.8,0,0}
\begin{tikzpicture}[scale=0.3][line cap=round,line join=round,>=triangle 45,x=1cm,y=1cm]
\clip(-11.90714257823503,-8.521586859071665) rectangle (8.282208622867046,11.938560408810387);
\draw [line width=0.4pt,color=ccqqqq] (1.2181800805352183,1.6940869099495046) circle (3.4390393467096096cm);
\draw [line width=0.4pt,color=ccqqqq] (-1.74,1.69) circle (9.374795997780433cm);
\draw [line width=0.4pt,color=qqwuqq] (0.9237346814822205,3.8254590283722223) circle (1.2874247901002658cm);
\draw [line width=0.4pt,color=qqwuqq] (-0.9587485739145309,1.0553473646319262) circle (1.1703378296922424cm);
\draw [line width=0.4pt,color=qqwuqq] (2.842100279247295,0.28475678550565386) circle (1.2888467480222292cm);
\draw [line width=0.4pt,color=qqwuqq] (-0.10665297235761506,-3.7422306366426854) circle (3.7023227979858664cm);
\draw [line width=0.4pt,color=qqwuqq] (4.779697020900073,3.626144718424509) circle (2.5736853814901486cm);
\draw [line width=0.4pt,color=qqwuqq] (-4.743874904406102,5.406073047853568) circle (4.5964633033874955cm);
\begin{scriptsize}
\draw[color=ccqqqq] (-2.0646807439511828,3.950076438777027) node {$\Omega_1$};
\draw[color=ccqqqq] (-5.016887101256826,11.201393338129341) node {$\Omega_2$};
\draw[color=qqwuqq] (1.1512455309725713,4.521757024203143) node {$\omega_1$};
\draw[color=qqwuqq] (-0.92485764768025,1.7536194526661597) node {$\omega_3$};
\draw[color=qqwuqq] (2.896375739115523,1.0314966079173815) node {$\omega_5$};
\draw[color=qqwuqq] (-3.211579989384807,-0.412749081580175) node {$\omega_4$};
\draw[color=qqwuqq] (3.919383102509667,5.454499032003649) node {$\omega_6$};
\draw[color=qqwuqq] (-5.167329360579495,9.486351581850993) node {$\omega_2$};
\end{scriptsize}
\end{tikzpicture}

        \smallskip
        \textbf{(b)}
    \end{minipage}

    \caption{Two tangency patterns for a circle tangent to an alternating triple: (a) external tangency and (b) internal tangency.}
    \label{opisane}
\end{figure}

In the classical Seven Circles configuration, $\Omega_1$ and $\Omega_2$ coincide with the common circle $\Omega$. Moreover, $S_i$ is the point at which $\omega_i$ touches $\Omega$. Viewed as a circle of radius zero, $S_i$ is tangent to $\omega_i$ and to both $\Omega_1$ and $\Omega_2$.

Motivated by this observation, we make the analogous construction for a general chain. For each $i$, let $\Gamma_i$ be a circle tangent to $\omega_i$, $\Omega_1$, and $\Omega_2$, and let $S_i$ denote its center; see Figure~\ref{final}. In the classical configuration, $\Gamma_i$ degenerates to a point circle, and its center $S_i$ becomes the corresponding point of tangency.

\begin{figure}[ht]
    \centering

    \begin{minipage}[t]{0.47\textwidth}
        \vspace{0pt}
        \centering
        \pgfplotsset{compat=1.15}
\definecolor{qqqqff}{rgb}{0,0,1}
\definecolor{ccqqqq}{rgb}{0.8,0,0}
\definecolor{wewdxt}{rgb}{0.43137254901960786,0.42745098039215684,0.45098039215686275}
\definecolor{qqwuqq}{rgb}{0,0.39215686274509803,0}
\begin{tikzpicture}[scale=0.37][line cap=round,line join=round,>=triangle 45,x=1cm,y=1cm]
\clip(-9.724120758729262,-7.8027266074507775) rectangle (7.591253005400042,10.240554588800405);
\draw [line width=0.4pt,color=qqwuqq] (0.9237346814822205,3.8254590283722223) circle (1.2874247901002658cm);
\draw [line width=0.4pt,color=qqwuqq] (-0.9587485739145309,1.0553473646319262) circle (1.1703378296922424cm);
\draw [line width=0.4pt,color=qqwuqq] (2.842100279247295,0.28475678550565386) circle (1.2888467480222292cm);
\draw [line width=0.4pt,color=qqwuqq] (-0.10665297235761506,-3.7422306366426854) circle (3.7023227979858664cm);
\draw [line width=0.4pt,color=qqwuqq] (4.779697020900073,3.626144718424509) circle (2.5736853814901486cm);
\draw [line width=0.4pt,color=qqwuqq] (-4.743874904406102,5.406073047853568) circle (4.5964633033874955cm);
\draw [line width=0.4pt,color=ccqqqq] (1.0655646966686945,1.6131537181838893) circle (0.9294221912440292cm);
\draw [line width=0.4pt,color=ccqqqq] (0.6797626122206704,1.8137042156415542) circle (1.908992694346952cm);
\draw [line width=0.4pt,color=qqqqff] (0.969269795008695,3.115189249988558) circle (0.5756968893902743cm);
\draw [line width=0.4pt,color=qqqqff] (-0.4395132594662024,1.198424407892322) circle (0.6317504625191979cm);
\draw [line width=0.4pt,color=qqqqff] (2.0297361797296145,0.8921985888670332) circle (0.2744893284429919cm);
\draw [line width=0.4pt,color=qqqqff] (-0.3229287889988095,2.4778386117803186) circle (0.7063021513500902cm);
\draw [line width=0.4pt,color=qqqqff] (2.0815528528701885,2.4333876435472703) circle (0.3763402779135702cm);
\draw [line width=0.4pt,color=qqqqff] (0.4751122412725761,0.36787309176683747) circle (0.4487497648677442cm);
\draw [line width=0.4pt,dash pattern=on 1pt off 1pt] (0.969269795008695,3.115189249988558)-- (0.4751122412725761,0.36787309176683747);
\draw [line width=0.4pt,dash pattern=on 1pt off 1pt] (-0.4395132594662024,1.198424407892322)-- (2.0815528528701885,2.4333876435472703);
\draw [line width=0.4pt,dash pattern=on 1pt off 1pt] (2.0297361797296145,0.8921985888670332)-- (-0.3229287889988095,2.4778386117803186);
\begin{scriptsize}
\draw[color=qqwuqq] (0.9739127133761044,4.693017643987236) node {$\omega_1$};
\draw[color=qqwuqq] (-2.5053702092739496,0.7534804833742224) node {$\omega_3$};
\draw[color=qqwuqq] (3.10559825318614,-0.516763837623848976) node {$\omega_5$};
\draw[color=qqwuqq] (-2.8641446559977797,-0.4685077837912195) node {$\omega_4$};
\draw[color=qqwuqq] (3.8193690389463977,5.553271881950312) node {$\omega_6$};
\draw[color=qqwuqq] (-5.665485379621247,9.43544485327086) node {$\omega_2$};
\draw[color=ccqqqq] (1.591531140631672,1.9137347213372986) node {$\Omega_1$};
\draw[color=ccqqqq] (0.225427271519339906,4.1077460225721675) node {$\Omega_2$};
\draw[color=qqqqff] (1.60188967930664492,3.9121853962900383) node {$\Gamma_1$};
\draw[color=qqqqff] (-1.5142489567096407,1.0269833086829397) node {$\Gamma_3$};
\draw[color=qqqqff] (2.922247523126296,0.6329024736401005) node {$\Gamma_5$};
\draw[color=qqqqff] (-1.5245379615767064,2.9725091680610842) node {$\Gamma_2$};
\draw[color=qqqqff] (2.997380771780892,2.5107751388587183) node {$\Gamma_6$};
\draw[color=qqqqff] (0.4004098880673631,-0.5146736626979983) node {$\Gamma_4$};
\draw [fill=qqqqff] (-0.3229287889988095,2.4778386117803186) circle (2.5pt);
\draw [fill=qqqqff] (0.969269795008695,3.115189249988558) circle (2.5pt);
\draw [fill=qqqqff] (2.0815528528701885,2.4333876435472703) circle (2.5pt);
\draw [fill=qqqqff] (2.0297361797296145,0.8921985888670332) circle (2.5pt);
\draw [fill=qqqqff] (0.4751122412725761,0.36787309176683747) circle (2.5pt);
\draw [fill=qqqqff] (-0.4395132594662024,1.198424407892322) circle (2.5pt);
\end{scriptsize}
\end{tikzpicture}

        \smallskip
        \textbf{(a)}
    \end{minipage}
    \hfill
    \begin{minipage}[t]{0.47\textwidth}
        \vspace{0pt}
        \centering
        \pgfplotsset{compat=1.15}
\definecolor{wewdxt}{rgb}{0.43137254901960786,0.42745098039215684,0.45098039215686275}
\definecolor{qqqqff}{rgb}{0,0,1}
\definecolor{qqwuqq}{rgb}{0,0.39215686274509803,0}
\definecolor{ccqqqq}{rgb}{0.8,0,0}
\begin{tikzpicture}[scale=0.3][line cap=round,line join=round,>=triangle 45,x=1cm,y=1cm]
\clip(-12.090457076643636,-8.432693824339331) rectangle (8.574098730913951,11.765042198408151);
\draw [line width=0.4pt,color=ccqqqq] (1.2181800805352183,1.6940869099495046) circle (3.4390393467096096cm);
\draw [line width=0.4pt,color=ccqqqq] (-1.74,1.69) circle (9.374795997780433cm);
\draw [line width=0.4pt,color=qqwuqq] (0.9237346814822205,3.8254590283722223) circle (1.2874247901002658cm);
\draw [line width=0.4pt,color=qqwuqq] (-0.9587485739145309,1.0553473646319262) circle (1.1703378296922424cm);
\draw [line width=0.4pt,color=qqwuqq] (2.842100279247295,0.28475678550565386) circle (1.2888467480222292cm);
\draw [line width=0.4pt,color=qqqqff] (0.3613315387059999,7.896469754577397) circle (2.822249847615249cm);
\draw [line width=0.4pt,color=qqqqff] (-6.258913348857249,-0.4997908028662068) circle (4.353266189564844cm);
\draw [line width=0.4pt,color=qqqqff] (5.1196565563742675,-1.6918358290926054) circle (1.7268092269797035cm);
\draw [line width=0.4pt,color=qqwuqq] (-0.10665297235761506,-3.7422306366426854) circle (3.7023227979858664cm);
\draw [line width=0.4pt,color=qqwuqq] (4.779697020900073,3.626144718424509) circle (2.5736853814901486cm);
\draw [line width=0.4pt,color=qqwuqq] (-4.743874904406102,5.406073047853568) circle (4.5964633033874955cm);
\draw [line width=0.4pt,color=qqqqff] (0.13025582500189836,-4.530148455285383) circle (2.879559000045656cm);
\draw [line width=0.4pt,color=qqqqff] (-5.070497076365206,5.810135097264576) circle (4.076897924450922cm);
\draw [line width=0.4pt,color=qqqqff] (5.72783759647236,3.9077126136554474) circle (1.584619645216099cm);
\draw [line width=0.4pt,dash pattern=on 1pt off 1pt] (0.13025582500189836,-4.530148455285383)-- (0.3613315387059999,7.896469754577397);
\draw [line width=0.4pt,dash pattern=on 1pt off 1pt] (-6.258913348857249,-0.4997908028662068)-- (5.72783759647236,3.9077126136554474);
\draw [line width=0.4pt,dash pattern=on 1pt off 1pt] (5.1196565563742675,-1.6918358290926054)-- (-5.070497076365206,5.810135097264576);
\begin{scriptsize}
\draw[color=ccqqqq] (-0.8867822412208483,3.9380304731061604) node {$\Omega_1$};
\draw[color=ccqqqq] (-4.83918975260611,11.407147030054382) node {$\Omega_2$};
\draw[color=qqwuqq] (1.1672248119399964,4.529335533864562) node {$\omega_1$};
\draw[color=qqwuqq] (-0.9179035602081339,1.7284168250089784) node {$\omega_3$};
\draw[color=qqwuqq] (2.941139994215271,1.0126264883014404) node {$\omega_5$};
\draw [fill=qqqqff] (0.3613315387059999,7.896469754577397) circle (2.5pt);
\draw[color=qqqqff] (-1.9137857678012704,10.629114055372275) node {$\Gamma_1$};
\draw [fill=qqqqff] (-6.258913348857249,-0.4997908028662068) circle (2.5pt);
\draw[color=qqqqff] (-9.020659686775226,1.9462660579199682) node {$\Gamma_3$};
\draw [fill=qqqqff] (5.1196565563742675,-1.6918358290926054) circle (2.5pt);
\draw[color=qqqqff] (5.524209470159969,-0.45007550410091957) node {$\Gamma_5$};
\draw[color=qqwuqq] (-3.252002484254548,-0.41895418511363525) node {$\omega_4$};
\draw[color=qqwuqq] (3.9370222018084076,5.4318537844958055) node {$\omega_6$};
\draw[color=qqwuqq] (-5.0881603045043935,9.508746571830041) node {$\omega_2$};
\draw[color=qqqqff] (-1.0112675171699903,-2.72193179017267) node {$\Gamma_4$};
\draw [fill=qqqqff] (0.13025582500189836,-4.530148455285383) circle (2.5pt);
\draw[color=qqqqff] (-8.138049565258376,7.454739518669282) node {$\Gamma_2$};
\draw [fill=qqqqff] (-5.070497076365206,5.810135097264576) circle (2.5pt);
\draw[color=qqqqff] (5.461966832185398,4.684942128800984) node {$\Gamma_6$};
\draw [fill=qqqqff] (5.72783759647236,3.9077126136554474) circle (2.5pt);
\draw[color=qqqqff] (0.461966832185398,8.684942128800984) node {$S_1$};
\draw[color=qqqqff] (-5.461966832185398,6.684942128800984) node {$S_2$};
\draw[color=qqqqff] (-5.861966832185398,-1.684942128800984) node {$S_3$};
\draw[color=qqqqff] (0.461966832185398,-5.684942128800984) node {$S_4$};
\draw[color=qqqqff] (5.461966832185398,-2.684942128800984) node {$S_5$};
\draw[color=qqqqff] (5.461966832185398,3.184942128800984) node {$S_6$};
\end{scriptsize}
\end{tikzpicture}

        \smallskip
        \textbf{(b)}
    \end{minipage}

    \caption{The general configuration in the two tangency cases: (a) external tangency and (b) internal tangency.}
    \label{final}
\end{figure}
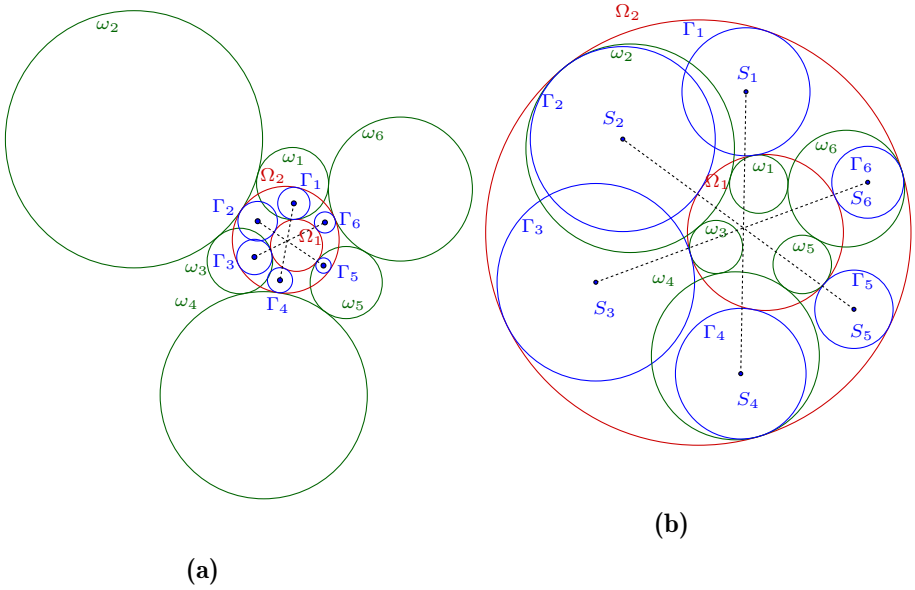

The classical theorem suggests asking whether the lines
\[
S_1S_4,\qquad S_2S_5,\qquad S_3S_6
\]
are concurrent. Our main theorem shows that they are.

\section{The Generalized Seven Circles Theorem}

We now make the tangency convention precise. An oriented circle is a circle equipped with one of its two possible orientations. Two oriented circles are said to be tangent if they are tangent in the usual Euclidean sense and their orientations agree at the point of tangency. Throughout this section, all circles are oriented, and all tangencies are understood in this sense.

We can now state the main result.

\begin{thm}\label{thm:main}
Let $\omega_1,\ldots,\omega_6$ form a closed chain of six consecutively tangent circles, indexed in cyclic order. Let $\Omega_1$ and $\Omega_2$ be circles with opposite orientations such that $\Omega_1$ is tangent to $\omega_1$, $\omega_3$, and $\omega_5$, while $\Omega_2$ is tangent to $\omega_2$, $\omega_4$, and $\omega_6$. For each $1\leq i\leq 6$, let $\Gamma_i$ be a circle centered at $S_i$ and tangent to $\Omega_1$, $\Omega_2$, and $\omega_i$. Then the lines
\[
S_1S_4,\qquad S_2S_5,\qquad S_3S_6
\]
are concurrent; see Figure~\ref{kuk}(a).
\end{thm}

The condition that $\Omega_1$ and $\Omega_2$ have opposite orientations corresponds, in the unoriented setting, to choosing the same tangency pattern for the two alternating triples: both circles are externally tangent to the members of their respective triples, or both are internally tangent. We added this assumption to avoid the situation in which our construction is undefined, that is, when the circles $\Gamma_i$ cannot be constructed. Accordingly, one can observe that in Figure~\ref{kuk}(a), if we took $\Omega_1$ to be a circle externally tangent to $\omega_1,\omega_3,\omega_5$, then the circles $\Omega_1$ and $\Omega_2$ would be directed the same way, with one inside the other, and there would not exist any (directed) circles $\Gamma_i$ tangent to both.

        \begin{figure}[ht]
    \centering

    \begin{minipage}[t]{0.47\textwidth}
        \vspace{0pt}
        \centering
        \pgfplotsset{compat=1.15}
\begin{center}
\definecolor{wewdxt}{rgb}{0.43137254901960786,0.42745098039215684,0.45098039215686275}
\definecolor{qqqqff}{rgb}{0,0,1}
\definecolor{qqwuqq}{rgb}{0,0.39215686274509803,0}
\definecolor{ccqqqq}{rgb}{0.8,0,0}
\begin{tikzpicture}[scale=0.32][line cap=round,line join=round,>=triangle 45,x=1cm,y=1cm]
\clip(-11.68053142174941,-8.15578439933071) rectangle (8.094579155645018,11.426633261352602);
\draw [line width=0.4pt,color=ccqqqq] (1.2181800805352183,1.6940869099495046) circle (3.4390393467096096cm);
\draw [line width=0.4pt,color=ccqqqq] (-1.74,1.69) circle (9.374795997780433cm);
\draw [line width=0.4pt,color=qqwuqq] (0.9237346814822205,3.8254590283722223) circle (1.2874247901002658cm);
\draw [line width=0.4pt,color=qqwuqq] (-0.9587485739145309,1.0553473646319262) circle (1.1703378296922424cm);
\draw [line width=0.4pt,color=qqwuqq] (2.842100279247295,0.28475678550565386) circle (1.2888467480222292cm);
\draw [line width=0.4pt,color=qqqqff] (0.3613315387059999,7.896469754577397) circle (2.822249847615249cm);
\draw [line width=0.4pt,color=qqqqff] (-6.258913348857249,-0.4997908028662068) circle (4.353266189564844cm);
\draw [line width=0.4pt,color=qqqqff] (5.1196565563742675,-1.6918358290926054) circle (1.7268092269797035cm);
\draw [line width=0.4pt,color=qqwuqq] (-0.10665297235761506,-3.7422306366426854) circle (3.7023227979858664cm);
\draw [line width=0.4pt,color=qqwuqq] (4.779697020900073,3.626144718424509) circle (2.5736853814901486cm);
\draw [line width=0.4pt,color=qqwuqq] (-4.743874904406102,5.406073047853568) circle (4.5964633033874955cm);
\draw [line width=0.4pt,color=qqqqff] (0.13025582500189836,-4.530148455285383) circle (2.879559000045656cm);
\draw [line width=0.4pt,color=qqqqff] (-5.070497076365206,5.810135097264576) circle (4.076897924450922cm);
\draw [line width=0.4pt,color=qqqqff] (5.72783759647236,3.9077126136554474) circle (1.584619645216099cm);
\draw [line width=0.4pt,dash pattern=on 1pt off 1pt] (0.13025582500189836,-4.530148455285383)-- (0.3613315387059999,7.896469754577397);
\draw [line width=0.4pt,dash pattern=on 1pt off 1pt] (-6.258913348857249,-0.4997908028662068)-- (5.72783759647236,3.9077126136554474);
\draw [line width=0.4pt,dash pattern=on 1pt off 1pt] (5.1196565563742675,-1.6918358290926054)-- (-5.070497076365206,5.810135097264576);
\draw [line width=0.4pt,color=qqqqff] (-9.754636556322438,2.337127918972759)-- (-9.794640841796108,2.0398071206021333);
\draw [line width=0.4pt,color=qqqqff] (-9.794640841796108,2.0398071206021333)-- (-9.51715133459684,2.153822792307003);
\draw [line width=0.4pt,color=qqwuqq] (1.4290701144376754,2.810294294452862)-- (1.2398944297756582,2.5774583858548623);
\draw [line width=0.4pt,color=qqwuqq] (1.2398944297756582,2.5774583858548623)-- (1.5361240838657662,2.5300453914582413);
\draw [line width=0.4pt,color=ccqqqq] (3.2685650699928805,4.455049246386487)-- (2.9693324564155255,4.4764931959217495);
\draw [line width=0.4pt,color=ccqqqq] (3.2685650699928805,4.455049246386487)-- (3.137519768259212,4.72491426615292);
\draw [line width=0.4pt,color=qqqqff] (5.7204675752389385,2.3231101074123788)-- (5.447050407029024,2.199647351376086);
\draw [line width=0.4pt,color=qqqqff] (5.7204675752389385,2.3231101074123788)-- (5.476837107985311,2.4981649428948214);
\draw [line width=0.4pt,color=qqwuqq] (5.903501714510298,1.3107787746662178)-- (5.606034714303026,1.349681008376887);
\draw [line width=0.4pt,color=qqwuqq] (5.903501714510298,1.3107787746662178)-- (5.721077891749263,1.0726159125545038);
\draw [line width=0.4pt,color=qqwuqq] (2.3812609969911716,1.488398383248577)-- (2.2206623856607357,1.235005238039925);
\draw [line width=0.4pt,color=qqwuqq] (2.3812609969911716,1.488398383248577)-- (2.0815167904304217,1.5007842878689117);
\draw [line width=0.4pt,color=qqwuqq] (-1.082681977735814,2.064398025300544)-- (-0.8249978319983333,2.2180172982956528);
\draw [line width=0.4pt,color=qqwuqq] (-0.8249978319983333,2.2180172982956528)-- (-1.0868780977917347,2.364368678159248);
\draw [line width=0.4pt,color=qqqqff] (-6.412513128600375,1.8033422928780294)-- (-6.120168661073041,1.870682565420947);
\draw [line width=0.4pt,color=qqqqff] (-6.120168661073041,1.870682565420947)-- (-6.3246592815566425,2.090190164683994);
\draw [line width=0.4pt,color=qqwuqq] (-4.2822927066202165,0.8328447817296212)-- (-4.531315918023257,0.6655503102653245);
\draw [line width=0.4pt,color=qqwuqq] (-4.2822927066202165,0.8328447817296212)-- (-4.551685574522509,0.9648579732044882);
\draw [line width=0.4pt,color=qqqqff] (3.1494426020598825,7.458827319742776)-- (3.242309870943839,7.173563108229977);
\draw [line width=0.4pt,color=qqqqff] (3.1494426020598825,7.458827319742776)-- (2.9488301825412395,7.23576979995279);
\draw [line width=0.4pt,color=ccqqqq] (4.390829468018523,8.782230258111742)-- (4.484737136968283,8.497306847854858);
\draw [line width=0.4pt,color=ccqqqq] (4.390829468018523,8.782230258111742)-- (4.68453421390876,8.721094979903972);
\draw [line width=0.4pt,color=qqqqff] (4.25854182402938,-3.3557149812988496)-- (4.556900055655637,-3.3243722688945363);
\draw [line width=0.4pt,color=qqqqff] (4.380577354676864,-3.0816578170801523)-- (4.556900055655637,-3.3243722688945363);
\draw [line width=0.4pt,color=qqwuqq] (-3.4774871945158226,-4.85889691882059)-- (-3.529062373410583,-5.154430336426735);
\draw [line width=0.4pt,color=qqwuqq] (-3.529062373410583,-5.154430336426735)-- (-3.7592142312773595,-4.961998212496073);
\draw [line width=0.4pt,color=qqqqff] (-2.466037411348625,-5.367097874256043)-- (-2.517232205389245,-5.662697422738694);
\draw [line width=0.4pt,color=qqqqff] (-2.517232205389245,-5.662697422738694)-- (-2.7476315267021207,-5.470561656316679);
\begin{scriptsize}
\draw[color=ccqqqq] (-1.7933816276175201,4.492259119067156) node {$\Omega_1$};
\draw[color=ccqqqq] (-5.177145482776821,11.00511750604884) node {$\Omega_2$};
\draw[color=qqwuqq] (0.9649412373639579,4.598078025431669) node {$\omega_1$};
\draw[color=qqwuqq] (-1.0583343880941811,1.8040307331324512) node {$\omega_3$};
\draw[color=qqwuqq] (2.7714373315230105,1.0814322954688602) node {$\omega_5$};
\draw [fill=qqqqff] (0.3613315387059999,7.896469754577397) circle (2.5pt);
\draw[color=qqqqff] (-1.7086729819914401,10.571558443450686) node {$\Gamma_1$};
\draw [fill=qqqqff] (-6.258913348857249,-0.4997908028662068) circle (2.5pt);
\draw[color=qqqqff] (-9.52735027533795,3.1375897957306057) node {$\Gamma_3$};
\draw [fill=qqqqff] (5.1196565563742675,-1.6918358290926054) circle (2.5pt);
\draw[color=qqqqff] (5.228272019579322,-0.38785119444710786) node {$\Gamma_5$};
\draw[color=qqwuqq] (-2.961176940608383,-0.46011103821346705) node {$\omega_4$};
\draw[color=qqwuqq] (3.8312483734296547,5.513369379805551) node {$\omega_6$};
\draw[color=qqwuqq] (-5.538444701608632,9.463574172366513) node {$\omega_2$};
\draw[color=qqqqff] (-1.0824210026829684,-2.579733122026667) node {$\Gamma_4$};
\draw [fill=qqqqff] (0.13025582500189836,-4.530148455285383) circle (2.5pt);
\draw[color=qqqqff] (-7.923019545898581,7.8497709949178285) node {$\Gamma_2$};
\draw [fill=qqqqff] (-5.070497076365206,5.810135097264576) circle (2.5pt);
\draw[color=qqqqff] (5.324618477934472,4.814857556730747) node {$\Gamma_6$};
\draw [fill=qqqqff] (5.72783759647236,3.9077126136554474) circle (2.5pt);


\draw[color=qqqqff] (0.324618477934472,8.414857556730747) node {$S_1$};
\draw[color=qqqqff] (-5.924618477934472,-0.014857556730747) node {$S_3$};
\draw[color=qqqqff] (5.324618477934472,-2.414857556730747) node {$S_5$};
\draw[color=qqqqff] (1.024618477934472,-4.414857556730747) node {$S_4$};
\draw[color=qqqqff] (6.324618477934472,4.414857556730747) node {$S_6$};
\draw[color=qqqqff] (-5.024618477934472,6.614857556730747) node {$S_2$};
\end{scriptsize}
\end{tikzpicture}
\end{center}

        \smallskip
        \textbf{(a)}
    \end{minipage}
    \hfill
    \begin{minipage}[t]{0.47\textwidth}
        \vspace{0pt}
        \centering
        \pgfplotsset{compat=1.15}
\begin{center}
\definecolor{wewdxt}{rgb}{0.43137254901960786,0.42745098039215684,0.45098039215686275}
\definecolor{qqwuqq}{rgb}{0,0.39215686274509803,0}
\definecolor{qqqqff}{rgb}{0,0,1}
\definecolor{ccqqqq}{rgb}{0.8,0,0}
\begin{tikzpicture}[scale=0.39][line cap=round,line join=round,>=triangle 45,x=1cm,y=1cm]
\clip(-12.410804906635116,-6.3562348648737155) rectangle (3.6824536117752,9.951886918134777);
\draw [line width=0.4pt,color=ccqqqq] (-4.395132699681245,1.7102540515308322) circle (7.757654694236227cm);
\draw [line width=0.4pt,color=qqqqff] (-7.638728785219712,-1.677910763642163) circle (3.067177492405188cm);
\draw [line width=0.4pt,color=qqqqff] (-8.681058006454752,3.4593749199708315) circle (3.1285533184124197cm);
\draw [line width=0.4pt,color=qqqqff] (-3.9163201333533455,4.89234531863769) circle (4.539741347691414cm);
\draw [line width=0.4pt,color=qqwuqq] (-7.120411118700901,-0.33496495388332326) circle (1.6276787423625458cm);
\draw [line width=0.4pt,color=qqwuqq] (-7.833825298983135,2.5611856880523614) circle (1.8938277792969553cm);
\draw [line width=0.4pt,color=qqwuqq] (-5.142775738473037,3.154889552484662) circle (2.4130194827857294cm);
\draw [line width=0.4pt,color=qqwuqq] (-2.9602986201480697,0.4783810489652778) circle (5.8665545038074685cm);
\draw [line width=0.4pt,color=qqwuqq] (-8.51876936299704,1.023446728840245) circle (3.5772140919140716cm);
\draw [line width=0.4pt,color=qqqqff] (-2.0726928060566387,-0.28367057421278247) circle (4.696696923758842cm);
\draw [line width=0.4pt,color=qqqqff] (-9.302158822667268,0.8929702437981555) circle (2.7830332821035153cm);
\draw [line width=0.4pt,color=qqwuqq] (-6.440696596215478,4.510533628340479) circle (4.289817131272597cm);
\draw [line width=0.4pt,color=qqqqff] (-6.752834887386835,4.937836086791347) circle (3.7606501407927198cm);
\draw [line width=0.4pt,dash pattern=on 1pt off 1pt] (-7.638728785219712,-1.677910763642163)-- (-6.752834887386835,4.937836086791347);
\draw [line width=0.4pt,dash pattern=on 1pt off 1pt] (-9.302158822667268,0.8929702437981555)-- (-3.9163201333533455,4.89234531863769);
\draw [line width=0.4pt,dash pattern=on 1pt off 1pt] (-2.0726928060566387,-0.28367057421278247)-- (-8.681058006454752,3.4593749199708315);
\draw [line width=0.4pt,color=qqqqff] (-3.9820360658394844,2.6248458464710516)-- (-4.042347332461691,2.330970786082425);
\draw [line width=0.4pt,color=qqqqff] (-4.042347332461691,2.330970786082425)-- (-3.7576884313153514,2.425677227247491);
\draw [line width=0.4pt,color=qqwuqq] (-5.0243319383553855,5.712857033250669)-- (-4.778913577911385,5.540317680183108);
\draw [line width=0.4pt,color=qqwuqq] (-4.778913577911385,5.540317680183108)-- (-5.051046221042426,5.414048822017258);
\draw [line width=0.4pt,color=qqqqff] (-7.38483682371861,6.139342760030762)-- (-7.085033433726901,6.150202195995483);
\draw [line width=0.4pt,color=qqqqff] (-7.085033433726901,6.150202195995483)-- (-7.244339676138974,6.404409829886636);
\draw [line width=0.4pt,color=ccqqqq] (-6.890510333933957,-5.475808376371037)-- (-7.18593152343484,-5.528022558294253);
\draw [line width=0.4pt,color=ccqqqq] (-7.18593152343484,-5.528022558294253)-- (-6.993002120701072,-5.757757722256627);
\draw [line width=0.4pt,color=qqqqff] (-7.049637762730039,-4.838869396654467)-- (-7.32895399147851,-4.729405044659183);
\draw [line width=0.4pt,color=qqqqff] (-7.32895399147851,-4.729405044659183)-- (-7.094496967467556,-4.542242270871384);
\draw [line width=0.4pt,color=qqqqff] (-4.4614566648465965,-4.499488079504036)-- (-4.607653402916432,-4.237521452282669);
\draw [line width=0.4pt,color=qqwuqq] (-9.134522351153,-2.6510038384529175)-- (-9.360239969804082,-2.453389171409924);
\draw [line width=0.4pt,color=qqwuqq] (-9.360239969804082,-2.453389171409924)-- (-9.076241838658905,-2.3567193130978557);
\draw [line width=0.4pt,color=qqqqff] (-4.607653402916432,-4.237521452282669)-- (-4.307685279764082,-4.241894676774455);
\draw [line width=0.4pt,color=qqwuqq] (-3.8538586447044856,-5.319723143670314)-- (-3.5771054428619,-5.203931670452835);
\draw [line width=0.4pt,color=qqwuqq] (-3.615203686435231,-5.501502710435935)-- (-3.8538586447044856,-5.319723143670314);
\draw [line width=0.4pt,color=qqqqff] (-9.683303769922855,-2.0131881980588027)-- (-9.914395144069273,-1.8218853575280005);
\draw [line width=0.4pt,color=qqqqff] (-9.914395144069273,-1.8218853575280005)-- (-9.633176337280267,-1.7174057771871498);
\draw [line width=0.4pt,color=qqwuqq] (-9.264517472932152,4.016563542668063)-- (-8.970462680605863,4.075992316086008);
\draw [line width=0.4pt,color=qqwuqq] (-8.970462680605863,4.075992316086008)-- (-9.066023249273318,3.791619009117913);
\draw [line width=0.4pt,color=qqwuqq] (-3.575830684522346,7.900709168540255)-- (-3.4812208189741143,7.616018154735913);
\draw [line width=0.4pt,color=qqwuqq] (-3.4812208189741143,7.616018154735913)-- (-3.7750754019319355,7.6764291146246855);
\draw [line width=0.4pt,color=qqwuqq] (-5.681058553744808,-0.013352921709970489)-- (-5.495019801053933,-0.24870283672422933);
\draw [line width=0.4pt,color=qqwuqq] (-5.495019801053933,-0.24870283672422933)-- (-5.384220172218514,0.030086406701569535);
\draw [line width=0.4pt,color=qqqqff] (-1.090276098174091,1.5347262869480118)-- (-1.199481007170077,1.2553085211363526);
\draw [line width=0.4pt,color=qqqqff] (-1.199481007170077,1.2553085211363526)-- (-0.9028956692104962,1.3004431786336907);
\begin{scriptsize}
\draw[color=ccqqqq] (-7.726785738392861,9.339526614029715) node {$\Omega_2$};
\draw [fill=ccqqqq] (-6.534335828563529,1.1835388632436916) circle (2.5pt);
\draw[color=ccqqqq] (-6.222742886204981,1.7333670472510514) node {$O_1$};
\draw[color=qqqqff] (-9.767986752076414,-0.26486131351283504) node {$\Gamma_1$};
\draw[color=qqqqff] (-10.011767975433644,5.809146012756571) node {$\Gamma_3$};
\draw[color=qqqqff] (-6.523551456642557,7.921429067681151) node {$\Gamma_5$};
\draw[color=qqwuqq] (-8.070566961750092,-0.39377927227179543) node {$\omega_1$};
\draw[color=qqwuqq] (-9.0374516524423,2.979573981921002) node {$\omega_3$};
\draw[color=qqwuqq] (-6.136797580365673,4.870370710385754) node {$\omega_5$};
\draw[color=qqwuqq] (-4.095596566682122,5.815769074618131) node {$\omega_6$};
\draw[color=qqwuqq] (-10.025822669594335,4.827398057466101) node {$\omega_2$};
\draw[color=qqqqff] (-4.203028198981256,3.4952458169568437) node {$\Gamma_6$};
\draw[color=qqqqff] (-10.75635776922845,2.7217380644030813) node {$\Gamma_2$};
\draw[color=qqwuqq] (-8.951506346602994,7.276839273886349) node {$\omega_4$};
\draw[color=qqqqff] (-8.349889205727841,7.899942741221325) node {$\Gamma_4$};
\draw [fill=qqqqff] (-2.0726928060566387,-0.28367057421278247) circle (2.5pt);
\draw[color=qqqqff] (-1.8395322884003025,0.25081052152300665) node {$S_6$};
\draw [fill=qqqqff] (-6.752834887386835,4.937836086791347) circle (2.5pt);
\draw[color=qqqqff] (-7.146654923977536,5.3860425454215966) node {$S_4$};
\draw [fill=qqqqff] (-9.302158822667268,0.8929702437981555) circle (2.5pt);
\draw[color=qqqqff] (-9.85393205791572,1.368099497433997) node {$S_2$};
\draw [fill=qqqqff] (-7.638728785219712,-1.677910763642163) circle (2.5pt);
\draw[color=qqqqff] (-6.931791659379267,-1.3606639629639985) node {$S_1$};
\draw [fill=qqqqff] (-8.681058006454752,3.4593749199708315) circle (2.5pt);
\draw[color=qqqqff] (-8.414348185107322,4.096862957831992) node {$S_3$};
\draw [fill=qqqqff] (-3.9163201333533455,4.89234531863769) circle (2.5pt);
\draw[color=qqqqff] (-3.687356363945412,5.42901519834125) node {$S_5$};
\end{scriptsize}
\end{tikzpicture}
\end{center}

        \smallskip
        \textbf{(b)}
    \end{minipage}

    \caption{The configuration of Theorem~\ref{thm:main}: (a) the original configuration in case of internal tangency and (b) the configuration after a dilatation sending
$\Omega_1$ to a point circle.}

    \label{kuk}
\end{figure}
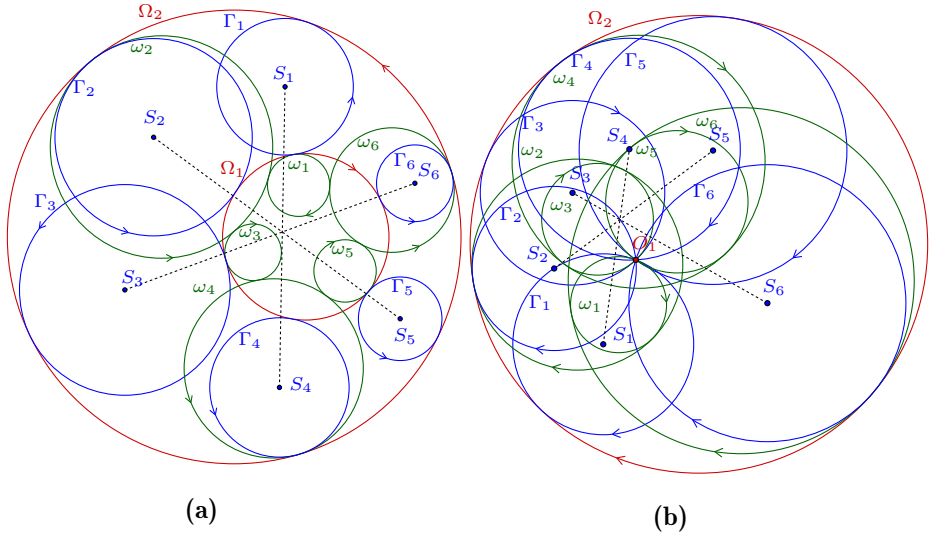

Before proceeding with the proof, we prove a following lemma. For the standard properties of poles and polars used below, see~\cite{chen}.

\begin{lem}\label{lemma1}
Let $A,B,C,D,E,F$ be six distinct points on a circle $\omega$, and suppose that the lines $AD$, $BE$, and $CF$ are concurrent. Let $P$ be a point in the plane, and let $A',B',C',D',E',F'$ denote the orthogonal projections of $P$ onto the tangents to $\omega$ at $A,B,C,D,E,F$, respectively. Then the circles
\[
(PA'D'),\qquad (PB'E'),\qquad (PC'F')
\]
either have a second common point distinct from $P$ or are tangent to one another at $P$.
\end{lem}

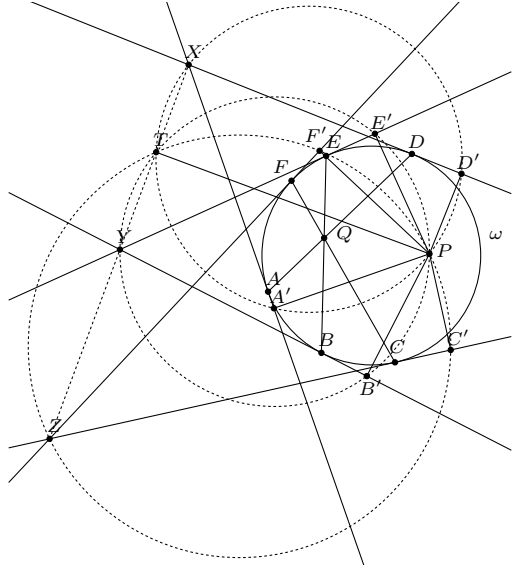
\begin{figure}[ht]
    \centering
    {
\begin{center}

\begin{tikzpicture}[scale=0.4][line cap=round,line join=round,>=triangle 45,x=1cm,y=1cm]
\clip(-14.311408692363601,-7.735286126914523) rectangle (2.3093543269357473,10.83969962268244);
\draw [line width=0.4pt] (-2.32,2.47) circle (3.616586655959378cm);
\draw [line width=0.4pt,domain=-14.311408692363601:2.3093543269357473] plot(\x,{(--18.047982596217707--3.413531462690265*\x)/-1.1947811486992483});
\draw [line width=0.4pt,domain=-14.311408692363601:2.3093543269357473] plot(\x,{(--25.310431192656193--2.638627794442613*\x)/2.4733261819780954});
\draw [line width=0.4pt,domain=-14.311408692363601:2.3093543269357473] plot(\x,{(--24.68298710171317--1.4957248677613277*\x)/3.2927961005843938});
\draw [line width=0.4pt,domain=-14.311408692363601:2.3093543269357473] plot(\x,{(--18.262246734901478-1.3417634719240952*\x)/3.3584773075716368});
\draw [line width=0.4pt,domain=-14.311408692363601:2.3093543269357473] plot(\x,{(--8.958572750827663--1.6499934556406353*\x)/-3.2182636057983984});
\draw [line width=0.4pt,domain=-14.311408692363601:2.3093543269357473] plot(\x,{(--2.5437296658023563-0.7820476671265066*\x)/-3.5310197516306});
\draw [line width=0.4pt] (-3.8157248677613276,5.762796100584394)-- (-3.969993455640635,-0.7482636057983982);
\draw [line width=0.4pt] (-5.733531462690265,1.2752188513007519)-- (-0.9782365280759046,5.828477307571637);
\draw [line width=0.4pt] (-1.5379523328734932,-1.0610197516306)-- (-4.958627794442613,4.943326181978096);
\draw [line width=0.4pt] (-5.541246879058617,0.725855081432508)-- (-0.3980074663591553,2.526056884376281);
\draw [line width=0.4pt] (-0.3980074663591553,2.526056884376281)-- (-2.470812761028924,-1.5168887521274033);
\draw [line width=0.4pt] (-0.3980074663591553,2.526056884376281)-- (0.6598941647167574,5.174018835909023);
\draw [line width=0.4pt] (-0.3980074663591553,2.526056884376281)-- (-4.03174682356708,5.932154064734805);
\draw [line width=0.4pt] (-0.3980074663591553,2.526056884376281)-- (-2.2013679520760743,6.496104119202485);
\draw [line width=0.4pt] (-0.3980074663591553,2.526056884376281)-- (0.30600393050724733,-0.6526216696333849);
\draw [line width=0.4pt,dash pattern=on 1pt off 1pt] (-5.5146080256293555,2.5964884832399786) circle (5.117085292746508cm);
\draw [line width=0.4pt,dash pattern=on 1pt off 1pt] (-4.378693351985692,5.651707260626379) circle (5.061180731269999cm);
\draw [line width=0.4pt,dash pattern=on 1pt off 1pt] (-6.677796506843572,-0.5320874317499927) circle (6.984840517237572cm);
\draw [line width=0.4pt,dash pattern=on 1pt off 1pt] (-12.95758554732799,-3.590231747876263)-- (-8.359379237612231,8.777357636876479);
\draw [line width=0.4pt] (-9.43445918599225,5.885763292957817)-- (-0.3980074663591553,2.526056884376281);
\begin{scriptsize}
\draw [fill=black] (-3.969993455640635,-0.7482636057983982) circle (2.5pt);
\draw[color=black] (-3.809928969955937,-0.32305748826026176) node {$B$};
\draw [fill=black] (-5.733531462690265,1.2752188513007519) circle (2.5pt);
\draw[color=black] (-5.566755464870001,1.6903841126524801) node {$A$};
\draw [fill=black] (-1.5379523328734932,-1.0610197516306) circle (2.5pt);
\draw[color=black] (-1.4609137688910645,-0.5401933471822242) node {$C$};
\draw [fill=black] (-3.88,3.05) circle (2.5pt);
\draw [fill=black] (-0.9782365280759046,5.828477307571637) circle (2.5pt);
\draw[color=black] (-0.8292458156635358,6.250237150013691) node {$D$};
\draw [fill=black] (-3.8157248677613276,5.762796100584394) circle (2.5pt);
\draw[color=black] (-3.573053487495614,6.250237150013691) node {$E$};
\draw [fill=black] (-4.958627794442613,4.943326181978096) circle (2.5pt);
\draw[color=black] (-5.310140358871317,5.421172961402561) node {$F$};
\draw [fill=black] (-0.3980074663591553,2.526056884376281) circle (2.5pt);
\draw[color=black] (0.0985164906393968,2.6378860424937702) node {$P$};
\draw [fill=black] (0.6598941647167574,5.174018835909023) circle (2.5pt);
\draw[color=black] (0.8881014321738077,5.598829573247803) node {$D'$};
\draw [fill=black] (-2.470812761028924,-1.5168887521274033) circle (2.5pt);
\draw[color=black] (-2.330759086887115,-1.8929028062563102) node {$B'$};
\draw [fill=black] (-4.03174682356708,5.932154064734805) circle (2.5pt);
\draw[color=black] (-4.145502570108062,6.4871126324740125) node {$F'$};
\draw [fill=black] (-2.2013679520760743,6.496104119202485) circle (2.5pt);
\draw[color=black] (-1.9741439808884316,6.921384350317937) node {$E'$};
\draw [fill=black] (0.30600393050724733,-0.6526216696333849) circle (2.5pt);
\draw[color=black] (0.5525278320216831,-0.22435937056846073) node {$C'$};
\draw [fill=black] (-5.541246879058617,0.725855081432508) circle (2.5pt);
\draw[color=black] (-5.310140358871318,1.1574142771167544) node {$A'$};
\draw [fill=black] (-9.43445918599225,5.885763292957817) circle (2.5pt);
\draw[color=black] (-9.277804690081732,6.309456020628771) node {$T$};
\draw [fill=black] (-8.359379237612231,8.777357636876479) circle (2.5pt);
\draw[color=black] (-8.192125395471916,9.211180680767722) node {$X$};
\draw [fill=black] (-12.95758554732799,-3.590231747876263) circle (2.5pt);
\draw[color=black] (-12.79145767990986,-3.165563277784133) node {$Z$};
\draw [fill=black] (-10.631208584899554,2.666920082103675) circle (2.5pt);
\draw[color=black] (-10.481921725921708,3.0918973838760553) node {$Y$};

\draw[color=black] (-3.210140358871318,3.1574142771167544) node {$Q$};

\draw[color=black] (1.810140358871318,3.1574142771167544) node {$\omega$};
\end{scriptsize}
\end{tikzpicture}

\end{center}
}
    \caption{The configuration of Lemma~\ref{lemma1}.}
    \label{fig:lemma1}
\end{figure}

\begin{proof}
Let $X$ be the intersection of the tangents to $\omega$ at $A$ and $D$, and define $Y$ and $Z$ analogously for the pairs $(B,E)$ and $(C,F)$.
Let $Q$ be the common point of the lines $AD$, $BE$, and $CF$.
By La Hire's theorem, the poles of these three lines lie on the polar $p$
of $Q$. Hence $X,Y,Z$ are collinear on $p$.

Let $T$ be the orthogonal projection of $P$ onto $p$. Since $C'$ and $F'$ are the orthogonal projections of $P$ onto the two tangents through $Z$, we have
\[
\arcangle PTZ=\arcangle PC'Z=\arcangle PF'Z=90^\circ.
\]
Thus $P,T,C',F',Z$ lie on the circle with diameter $PZ$, and consequently
\[
T\in(PC'F').
\]
Applying the same argument to $X$ and $Y$, we obtain
\[
T\in(PA'D')\cap(PB'E').
\]
Therefore, all three circles pass through $P$ and $T$.

If $P\neq T$, then $T$ is a second common point distinct from $P$. If $P=T$, then $P\in p$. Since the three circles have diameters $PX$, $PY$, and $PZ$, their centers lie on $p$, and hence they have the same tangent at $P$, namely the line through $P$ perpendicular to $p$. Thus, they are tangent to one another at $P$.
\end{proof}

We now introduce the main tool used in the proof, a transformation called a dilatation. We assign a signed radius to each oriented circle, taking the radius to be positive for counterclockwise orientation and negative for clockwise orientation. A dilatation with parameter $c$ replaces the signed radius $r$ of every circle by $r+c$, while leaving its center unchanged. The transformation preserves tangency of oriented circles. By choosing $c$ appropriately --- specifically, $c=-r$, where $r$ is the radius of a chosen circle --- we can transform a chosen circle into a point circle. For further background on dilatations and proofs of these properties, see~\cite{laguerre}.
Now we are ready to prove our main result.

\begin{proof}[Proof of Theorem~\ref{thm:main}]
    Let us first apply a dilatation that transforms the circle $\Omega_1$ into a point, as shown in Figure~\ref{kuk}(b). Denote this point by $O_1$, and apply an inversion centered at $O_1$, as shown in Figure~\ref{after_inv}. From this point onward, all objects are understood to be their images under this inversion.

         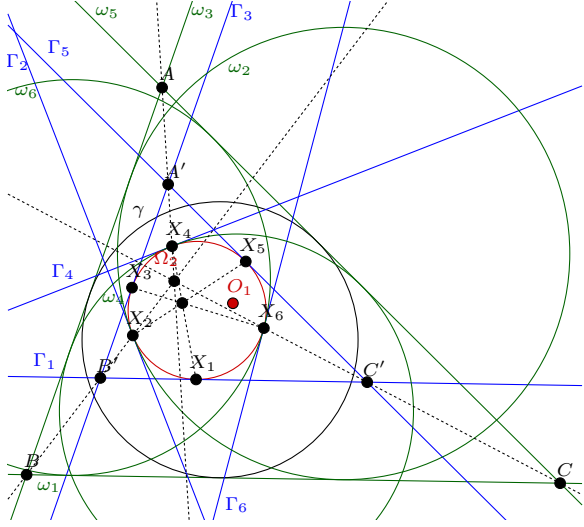
\begin{figure}[H]
        \centering
        \pgfplotsset{compat=1.15}
\begin{center}
\definecolor{qqqqff}{rgb}{0,0,1}
\definecolor{ccqqqq}{rgb}{0.8,0,0}
\definecolor{qqwuqq}{rgb}{0,0.39215686274509803,0}
\begin{tikzpicture}[scale=0.8][line cap=round,line join=round,>=triangle 45,x=1cm,y=1cm]
\clip(-4.528894067095636,-1.6645266498159859) rectangle (4.962813322477587,6.891984201939667);
\draw [line width=0.4pt,color=qqwuqq,domain=-4.528894067095636:4.962813322477587] plot(\x,{(-24.9248-6.4*\x)/-2.24});
\draw [line width=0.4pt,color=qqwuqq,domain=-4.528894067095636:4.962813322477587] plot(\x,{(-8.7934-0.14*\x)/8.82});
\draw [line width=0.4pt,color=qqwuqq,domain=-4.528894067095636:4.962813322477587] plot(\x,{(-23.0434--6.54*\x)/-6.58});
\draw [line width=0.4pt] (-1.0219358793185966,1.3010242200945619) circle (2.281499745812552cm);
\draw [line width=0.4pt,color=ccqqqq] (-1.4009679396592982,1.785512110047281) circle (1.1407498729062762cm);
\draw [line width=0.4pt,color=qqqqff,domain=-4.528894067095636:4.962813322477587] plot(\x,{(--3.482606527563674--1.0767063000660613*\x)/0.3768472050231211});
\draw [line width=0.4pt,color=qqqqff,domain=-4.528894067095636:4.962813322477587] plot(\x,{(--3.7800809896937597--0.41515919558255554*\x)/1.062522053822381});
\draw [line width=0.4pt,color=qqqqff,domain=-4.528894067095636:4.962813322477587] plot(\x,{(--1.6193307478856092-0.8041689895645884*\x)/0.8090874543325672});
\draw [line width=0.4pt,color=qqqqff,domain=-4.528894067095636:4.962813322477587] plot(\x,{(-0.764023128248946-1.1029217595468621*\x)/-0.29133119443983024});
\draw [line width=0.4pt,color=qqqqff,domain=-4.528894067095636:4.962813322477587] plot(\x,{(-0.7098915684352782--0.0181048601998286*\x)/-1.1406061925891993});
\draw [line width=0.4pt,color=qqqqff,domain=-4.528894067095636:4.962813322477587] plot(\x,{(--2.0536383636948337--1.0633643324049913*\x)/-0.4129970570168411});
\draw [line width=0.4pt,color=qqwuqq] (-3.4608188113277127,2.3296112808331007) circle (3.2712497018579105cm);
\draw [line width=0.4pt,color=qqwuqq] (1.020877577399348,2.726125821278767) circle (3.7388433782008277cm);
\draw [line width=0.4pt,color=qqwuqq] (-0.7509905241252501,0.1220168873958321) circle (2.926719356544928cm);
\draw [line width=0.4pt,dash pattern=on 1pt off 1pt,domain=-4.528894067095636:4.962813322477587] plot(\x,{(-4.295107785311008-2.6219658361303377*\x)/0.16387286475814622});
\draw [line width=0.4pt,dash pattern=on 1pt off 1pt,domain=-4.528894067095636:4.962813322477587] plot(\x,{(--8.083842536808245--2.30251505303044*\x)/1.7556677279357102});
\draw [line width=0.4pt,dash pattern=on 1pt off 1pt,domain=-4.528894067095636:4.962813322477587] plot(\x,{(-4.571245598147934--1.7883618312149017*\x)/-3.416092360224872});
\draw [line width=0.4pt,dash pattern=on 1pt off 1pt] (-1.4190727998591268,0.6449059174580818)-- (-1.8161271352418538,2.848034163869662);
\draw [line width=0.4pt,dash pattern=on 1pt off 1pt] (-2.4643322720642895,1.37251505303044)-- (-0.5967989500947098,2.5945995643798483);
\draw [line width=0.4pt,dash pattern=on 1pt off 1pt] (-0.2980461801124361,1.4941809156074508)-- (-2.4776742397253595,2.162359315070402);
\begin{scriptsize}
\draw [fill=black] (-1.98,5.47) circle (2.5pt);
\draw[color=black] (-1.8952946215175612,5.704177105056501) node {$A$};
\draw [fill=black] (-4.22,-0.93) circle (2.5pt);
\draw[color=black] (-4.131166803885885,-0.7024566482680827) node {$B$};
\draw[color=qqwuqq] (-1.304078419449014,6.714619341319104) node {$\omega_3$};
\draw [fill=black] (4.6,-1.07) circle (2.5pt);
\draw[color=black] (4.683329299681547,-0.8421986596661022) node {$C$};
\draw[color=qqwuqq] (-3.851682781089845,-1.1324320679542963) node {$\omega_1$};
\draw[color=qqwuqq] (-2.884238086795859,6.714619341319104) node {$\omega_5$};
\draw[color=black] (-2.368267583172399,3.4145579952274137) node {$\gamma$};
\draw [fill=black] (-1.8161271352418538,2.848034163869662) circle (2.5pt);
\draw[color=black] (-1.7018056826587642,3.1135752014470643) node {$X_4$};
\draw [fill=black] (-2.4643322720642895,1.37251505303044) circle (2.5pt);
\draw[color=black] (-2.3467688121880883,1.640909389021782) node {$X_2$};
\draw [fill=black] (-0.2980461801124361,1.4941809156074508) circle (2.5pt);
\draw[color=black] (-0.17539294277269651,1.7591526294354909) node {$X_6$};
\draw[color=ccqqqq] (-1.9060440070097169,2.640602239792229) node {$\Omega_2$};
\draw [fill=black] (-2.4776742397253595,2.162359315070402) circle (2.5pt);
\draw[color=black] (-2.357518197680244,2.4256145299491223) node {$X_3$};
\draw [fill=black] (-0.5967989500947098,2.5945995643798483) circle (2.5pt);
\draw[color=black] (-0.4763757365530478,2.855589949635336) node {$X_5$};
\draw [fill=black] (-1.4190727998591268,0.6449059174580818) circle (2.5pt);
\draw[color=black] (-1.304078419449014,0.9099511755552188) node {$X_1$};
\draw[color=qqqqff] (-0.6591152899196897,6.714619341319104) node {$\Gamma_3$};
\draw[color=qqqqff] (-3.615196300262426,2.4256145299491223) node {$\Gamma_4$};
\draw[color=qqqqff] (-3.6796926132153587,6.12340313925056) node {$\Gamma_5$};
\draw[color=qqqqff] (-0.7451103738569329,-1.4226654762424906) node {$\Gamma_6$};
\draw[color=qqqqff] (-3.9161790940427776,0.9421993320316847) node {$\Gamma_1$};
\draw[color=qqqqff] (-4.367653284713304,5.886916658423143) node {$\Gamma_2$};
\draw [fill=ccqqqq] (-0.8096066868098687,1.904269333579587) circle (2.5pt);
\draw[color=ccqqqq] (-0.6913634463961559,2.167629278137394) node {$O_1$};
\draw[color=qqwuqq] (-4.217161887823129,5.392444925783997) node {$\omega_6$};
\draw[color=qqwuqq] (-0.6913634463961559,5.768673418009434) node {$\omega_2$};
\draw[color=qqwuqq] (-2.7767442318743045,1.9741403392785977) node {$\omega_4$};
\draw [fill=black] (-1.78,2.27) circle (2.5pt);
\draw [fill=black] (-1.88,3.87) circle (2.5pt);
\draw[color=black] (-1.7555526101195411,4.102518666725356) node {$A'$};
\draw [fill=black] (-3,0.67) circle (2.5pt);
\draw[color=black] (-2.873488701303703,0.8992017900630633) node {$B'$};
\draw [fill=black] (1.41,0.6) circle (2.5pt);
\draw[color=black] (1.5337593504800129,0.8347054771101313) node {$C'$};
\draw [fill=black] (-1.646639371790485,1.9076004604549774) circle (2.5pt);
\end{scriptsize}
\end{tikzpicture}
\end{center}
        \caption{After inversion.}
        \label{after_inv}
    \end{figure}
    
    Denote the points of tangency of the lines $\Gamma_1$, $\Gamma_2$, $\Gamma_3$, $\Gamma_4$, $\Gamma_5$, and $\Gamma_6$ with the circle $\Omega_2$ by $X_1$, $X_2$, $X_3$, $X_4$, $X_5$, and $X_6$, respectively. Moreover, let
    $A=\omega_3\cap\omega_5$,
    $B=\omega_3\cap\omega_1$,
    $C=\omega_1\cap\omega_5$,
    $A'=\Gamma_3\cap\Gamma_5$,
    $B'=\Gamma_3\cap\Gamma_1$, and
    $C'=\Gamma_1\cap\Gamma_5$.

    
    The points $S_1$, $S_2$, $S_3$, $S_4$, $S_5$, and $S_6$ are the reflections of $O_1$ across the lines $\Gamma_1$, $\Gamma_2$, $\Gamma_3$, $\Gamma_4$, $\Gamma_5$, and $\Gamma_6$.
    
    Under the inversion, the desired concurrence is equivalent to showing that the circles
    \[
    (O_1S_1S_4),\qquad (O_1S_2S_5),\qquad (O_1S_3S_6)
    \]
    have a second common point or are tangent to one another at $O_1$.

    For each $i$, the point $S_i$ is the reflection of $O_1$ across the line
    $\Gamma_i$. Hence the homothety centered at $O_1$ with ratio $1/2$ sends $S_i$ to the orthogonal projection of $O_1$ onto $\Gamma_i$. Applying Lemma~\ref{lemma1} with
    \[
    (A,B,C,D,E,F)=(X_1,X_6,X_5,X_4,X_3,X_2),
    \]
    it is therefore enough to prove that the lines
    \[
    X_1X_4,\qquad X_3X_6,\qquad X_2X_5
    \]
    are concurrent.

    Observe that, since $X_4$ is a center of homothety of the circles $\Omega_2$ and $\omega_4$, and since $AB\parallel A'B'$ and $A'C'\parallel AC$, this homothety maps $A'$ to $A$. Therefore, the points $A$, $A'$, and $X_4$ are collinear. Analogously, the points $B$, $B'$, and $X_2$ are collinear, as are the points $C$, $C'$, and $X_6$.
    
    Let $\gamma$ be the incircle of the triangle $ABC$. By Monge's theorem~\cite{chen}, applied to the circles $\gamma$, $\omega_4$, and $\Omega_2$, the positive center of homothety of $\gamma$ and $\Omega_2$ lies on the line $AX_4$, since the center of positive homothety of the circles $\Omega_2$ and $\omega_4$ is $X_4$, and the center of positive homothety of the circles $\gamma$ and $\omega_4$ is $A$. Analogously, it lies on the lines $BX_2$ and $CX_6$. Thus, the lines $AX_4$, $BX_6$, and $CX_2$ are concurrent. Consequently, the lines $A'X_4$, $B'X_2$, and $C'X_6$ are concurrent. By Steinbart's theorem~\cite{st}, the lines $X_1X_4$, $X_3X_6$, and $X_2X_5$ are therefore concurrent, which completes the proof.
\end{proof}


\section*{Acknowledgements}
The author would like to thank Konstanty Smolira for drawing his attention to this problem.

\end{document}